\documentclass[12pt]{amsart}
\usepackage[top=1.5in, bottom=1.5in, left=1.4in, right=1.4in]{geometry}  
\usepackage{geometry}                % See geometry.pdf to learn the layout options. There are lots.
\usepackage{graphicx}
\usepackage{amssymb}
\usepackage{epstopdf}
\usepackage{color}
\usepackage{url}
\usepackage{verbatim}
\usepackage{enumitem}
\usepackage{xcolor}

\definecolor{darkbrown}{rgb}{.5,.1,.1} 
\title[On the Number of Circuit-cocircuit Reversal Classes]{On the Number of Circuit-cocircuit Reversal Classes of an Oriented Matroid}

\author{Emeric Gioan}
\address{Emeric Gioan: CNRS, LIRMM, Universit\'e de Montpellier, France}
\email{gioan@lirmm.fr}

\author{Chi Ho Yuen}
\address{Chi Ho Yuen:
School of Mathematics, Georgia Institute of Technology\\
Atlanta, Georgia 30332-0160, USA
}
\email{cyuen7@math.gatech.edu}

\date{\today}  % Activate to display a given date or no date

\numberwithin{equation}{section}
\theoremstyle{definition}

\newtheorem{theorem}{Theorem}[section]

\newtheorem{corollary}[theorem]{Corollary}

\newtheorem{proposition}[theorem]{Proposition}

\DeclareRobustCommand{\rchi}{{\mathpalette\irchi\relax}}
\newcommand{\irchi}[2]{\raisebox{\depth}{$#1\chi$}}

\begin{document}

\begin{abstract}
The first author introduced the circuit-cocircuit reversal system of an oriented matroid, and showed that when the underlying matroid is regular, the cardinalities of such system and its variations are equal to special evaluations of the Tutte polynomial (e.g., the total number of circuit-cocircuit reversal classes equals $t(M;1,1)$, the number of bases of the matroid). 
By relating these classes to activity classes studied by the first author and Las Vergnas, 
we give an alternative proof of the above results and a proof of the converse statements that these equalities fail whenever the underlying matroid is not regular. Hence we extend the above results to an equivalence of matroidal properties, thereby giving a new characterization of regular matroids.
\end{abstract}

\maketitle

\vspace{-5mm}
\section{Introduction}

The {\em cycle-cocycle reversal system} of a graph, introduced in \cite{gioan2007enumerating}, consists of equivalence classes of orientations of the graph with respect to the {\em cycle-cocycle reversal relation}, that is, two orientations are equivalent if they differ by successive reversals of directed (co)cycles.
As proven in the same paper, the number of cycle-cocycle reversal classes of $G$ equals the evaluation $t(G;1,1)$ of its Tutte polynomial, which is also the number of spanning trees of $G$.
This result can be thought as a ``linear algebra free'' formulation of Kirchhoff's Matrix-Tree Theorem.
In particular, the cycle-cocycle reversal system is related to various combinatorial objects associated to the graph Laplacian, notably the {\em sandpile group} (also known as the {\em critical group} or {\em Jacobian group} in the literature).
The notion of {\em chip-firing} (related to the {\em abelian sandpile model}) can also be partially interpreted by {\em cycle-cocycle reversals}.
We refer the reader to \cite[Section 5]{gioan2007enumerating} and \cite{backman2014riemann} for details.

\smallskip

The above setup, and further results involving Tutte polynomial evaluations for related reversal classes, were generalized to {\em regular matroids} in \cite{gioan2008circuit} (in terms of a {\em circuit-cocircuit reversal system}).
This provides an approach to generalize the theory of sandpile groups and chip-firing to regular matroids (different from the approach in \cite{MerinoDissertation}).
Such topic has been investigated since then in \cite{bby2017geometric} in a unifying way.

\smallskip

The circuit-cocircuit reversal system is actually defined for general oriented matroids, and it was shown in \cite{gioan2008circuit} that the aforementioned Tutte polynomial evaluations are not available for $U_{2,k}$, $k\geq 4$.
Since $U_{2,4}$ is the excluded minor for the class of regular matroids within oriented matroids, it was expected that these enumerative results are not available when the oriented matroid is not regular; we will prove this rigorously in this note.
In particular, we extend the results in \cite{gioan2008circuit} to an equivalence of (oriented) matroid properties.
We use a noteworthy general relation between circuit-cocircuit reversal classes and activity classes of (re)orientations, which are known to be enumerated using the same Tutte polynomial evaluations, and have been introduced in the context of {\em active bijections} \cite{gioan2002thesis,gioan2005activity, gioanlasvergans2018AB, gioanlasvergans2018AB2}.
We also use this to give a short proof of results in~\cite{gioan2008circuit}.

\section{Preliminaries}

We assume that the reader is familiar with the basic theory of oriented matroids \cite{bjorner1999oriented}.
Given an oriented matroid $M$ on $E$, we identify the set of its reorientations with $2^E$ via the bijection associating $A\subseteq E$ with $_{-A}M$.

Let $M$ be an oriented matroid on $E$.
Following \cite{gioan2008circuit}, let us write $M \sim -_CM$ if $C$ is a positive circuit or cocircuit of $M$ (we say that $-_CM$ is obtained from $M$ by a {\em circuit or cocircuit reversal}, respectively).
Applying the same rule to reorientations $-_{A}M$ for $A\subseteq E$ (i.e., writing $-_{A}M\sim-_{C\triangle A}M$ when $C$ is a positive circuit or cocircuit of $-_{A}M$) and taking the transitive closure of the relation, we obtain an equivalence relation, whose equivalence classes are called {\em circuit-cocircuit reversal classes} of reorientations of~$M$.
Allowing only the use of positive circuits, resp. positive cocircuits, yields by the same way the \emph{circuit reversal classes}, resp. the \emph{cocircuit reversal classes}.
As observed in \cite{gioan2008circuit}, circuit reversals act on the totally cyclic part of $M$ (the union of positive circuits of $M$) and cocircuit reversals act on the acyclic part of $M$ (the union of positive cocircuits of $M$).

It was shown and geometrically illustrated in \cite[Proposition 2 and Figure 1]{gioan2008circuit} that, for any integer $k$, the number of acyclic cocircuit reversal classes of a uniform oriented matroid $U_{2,k}$ equals $1$, or $2$, if $k$ is even, or odd, respectively.
On the other hand, it is well-known that an oriented matroid $M$ is regular if and only if $U_{2,4}$ is not a minor of $M$. Indeed, regular matroids are precisely the orientable binary matroids \cite[Theorem 7.9.3]{bjorner1999oriented}, so the claim follows from \cite[Theorem 6.5.4]{oxley2006matroid}.

\medskip

Now, let $M$ be an oriented matroid on a linearly ordered set $E$.
Consider a reorientation $-_AM$ of $M$ such that $A$ does not contain the minimum element of a positive circuit or cocircuit of $-_AM$, we call such a reorientation \emph{circuit-cocircuit minimal} (with respect to $M$); the terminology here is from \cite{backman2014riemann}, and it is called \emph{active fixed and dual-active fixed} in \cite{gioanlasvergans2018AB, gioanlasvergans2018AB2}.
Similarly, we can define a \emph{circuit minimal} (or \emph{active-fixed}), resp. a \emph{cocircuit minimal}  (or \emph{dual-active-fixed}),  reorientation $-_AM$ of $M$ when $A$ does not contain the minimum element of a positive circuit, resp. cocircuit.
Let us denote by $t(M;x,y)$ the Tutte polynomial of $M$.
From the works on {\em active bijections} \cite{gioan2002thesis,gioan2005activity, gioanlasvergans2018AB, gioanlasvergans2018AB2}, we have:

\begin{theorem} \label{CCMO_main}
Let $M$ be an oriented matroid on a linearly ordered set.
Then
 \begin{enumerate}[leftmargin=10mm]
 \item $t(M;1,1)=\#$ circuit-cocircuit minimal reorientations of $M$,
 \item $t(M;1,2)=\#$ cocircuit minimal reorientations of $M$,
 \item $t(M;2,1)=\#$ circuit minimal reorientations of $M$,
 \item $t(M;1,0)=\#$ (circuit-)cocircuit minimal acyclic reorientations of $M$,
 \item $t(M;0,1)=\#$ circuit(-cocircuit) minimal totally cyclic reorientations of~$M$.
 \end{enumerate}
\end{theorem}

Let us explain this briefly; details can be found in \cite{gioan2002thesis,gioan2005activity, gioanlasvergans2018AB, gioanlasvergans2018AB2}.
The {\em active partition} of $M$ is a partition of its ground set induced by taking differences of unions of positive circuits/cocircuits whose the minimum element is greater than a given element.
Therefore, the minimum elements of the parts are precisely the minimum elements of some positive circuits/cocircuits (called {\em active/dual-active elements}).
Activity classes of reorientations are the sets of reorientations obtained from a given reorientation by arbitrarily reorienting parts of its active partition.
It turns out that all reorientations obtained by this way share the same active partition. Hence activity classes partition the set of reorientations.
By choosing a suitable reorientation for each part, each activity class contains a unique circuit-cocircuit minimal reorientation, which can be thought of as a representative of the class.
Finally, using a classical formula of the Tutte polynomial in terms of orientation activities \cite{lasvergnas1984activity}, one enumerates activity classes and gets the above evaluations.

\section{Results}

We first give a noteworthy property relating reversal classes and activity classes.

\begin{proposition} \label{prop:CCMO_rep}
Let $M$ be an oriented matroid on a linearly ordered ground set.
Every circuit-cocircuit reversal class of $M$ contains at least one circuit-cocircuit minimal  reorientation.
\end{proposition}

The following proof is essentially given in \cite{backman2018partial} for graphs, but for the sake of interest and completeness, we include it here.
A corollary of the proof is that a minimal reorientation can be obtained greedily. Moreover, we note that there is an interpretation using combinatorial commutative algebra \cite[Section 4]{backman2018partial}.

\begin{proof}
Start with an arbitrary reorientation of $M$, and greedily reorient any positive (co)circuit whose minimal element is in the set of reoriented elements with respect to $M$.
Once the procedure stops, we will have a circuit-cocircuit minimal reorientation equivalent to the starting reorientation, so it suffices to show the procedure always terminates.
If this is not the case, then, since the number of reorientations is finite, without loss of generality, we must return to the starting reorientation.
Let $e$ be the minimal element that was reoriented (which must occur at least twice) in the process.
When $e$ was reoriented for the first time, we must have reoriented it to remove it from the set of reoriented elements with respect to $M$, so the second reorientation is not valid, a contradiction.
\end{proof}

By combining Theorem \ref{CCMO_main} and Proposition \ref{prop:CCMO_rep}, we can use the set of circuit-cocircuit minimal reorientations (and variations thereof) as an intermediate object, and get the following corollary concerning the enumeration of reversal classes in terms of the Tutte polynomial.

\begin{corollary} \label{coro:CCMO_red}
Let $M$ be an oriented matroid on a linearly ordered ground set.
The number of circuit-cocircuit reversal classes is at most $t(M;1,1)$, with equality if and only if no two circuit-cocircuit minimal reorientations are contained in the same class.
The number of acyclic cocircuit reversal classes is at most $t(M ; 1, 0)$, with equality if and only if no two acyclic cocircuit minimal reorientations are contained in the same class.
Analogous statements hold for each of the other settings in Theorem \ref{CCMO_main}.
\end{corollary}

\begin{proof}
The first statement follows from comparing Equation~(1) of Theorem \ref{CCMO_main} and Proposition \ref{prop:CCMO_rep}.
The variations follow from comparing the other equations of Theorem \ref{CCMO_main} and the corresponding counterparts of Proposition \ref{prop:CCMO_rep}, since circuit and cocircuit reversals preserve the acyclic and totally cyclic parts of reorientations.
\end{proof}

\medskip

Now we prove the main theorem of this note, which includes the original results of \cite{gioan2008circuit} (the enumerations when $M$ is regular, with a new proof) and their converses.

\begin{theorem} \label{thm}
Let $M$ be an oriented matroid. Consider the following six statements:
\begin{enumerate}
\item $M$ is regular,
\item $t(M;1,1)=\#$ circuit-cocircuit reversal classes of $M$,
\item $t(M;1,2)=\#$ cocircuit reversal classes of $M$,
\item $t(M;2,1)=\#$ circuit reversal classes of $M$,
\item $t(M;1,0)=\#$ acyclic (circuit-)cocircuit reversal classes of $M$,
\item $t(M;0,1)=\#$ totally cyclic circuit(-cocircuit) reversal classes of $M$.
\end{enumerate}

Then we have the following implications: (1) implies all other statements; (2), (3), (4) each implies (1); (5) implies (1) if $M$ has no loops; (6) implies (1) if $M$ has no coloops.
In particular, if $M$ has no loops nor coloops, then all statements are equivalent. Moreover, if any of the equalities fail, then the left hand side is larger.
\end{theorem}

\begin{proof} Let us separate implications.

\smallskip

$\bullet$ $(1)\Rightarrow (2)$. We give an alternative proof to that of \cite{gioan2008circuit}.
By Corollary \ref{coro:CCMO_red}, it suffices to show that every circuit-cocircuit reversal class contains a unique circuit-cocircuit minimal reorientation.
We claim that any two reorientations within a reversal class differ by a disjoint union of positive circuits and cocircuits, which will imply that at most one of them can be minimal, thus proving the implication.

By induction and restricting to the totally circuit part (the acyclic part follows from duality), it suffices to show that if $C$ is a positive circuit of $M$ and $D$ is a positive circuit of $-_{C}M$, then $C\triangle D$ is a disjoint union of positive circuits of $M$.
By \cite[Corollary 7.9.4]{bjorner1999oriented}, we may assume that some totally unimodular matrix $Q$ realizes $M$.
By total unimodularity, if $C$ is a positive circuit of $M$, then the characteristic vector $\rchi_C\in\{0,1\}^E$ of $C$ is in the kernel $\ker(Q)$ of $Q$; similarly,  since $D$ is a positive circuit of $-_CM$, $\rchi_{D\setminus C}-\rchi_{C\cap D}\in\ker(Q)$.
Since their sum $\rchi_{C\triangle D}$ is in $\ker(Q)$, $C\triangle D$ is a positive vector and contains some positive circuit $C_1$ of $M$, thus $\rchi_{(C\triangle D)\setminus C_1}=\rchi_{C\triangle D}-\rchi_{C_1}\in\ker(Q)$.
Proceeding by induction, we can write $C\triangle D$ as a disjoint union of positive circuits.

\smallskip

$\bullet$ $(1)\Rightarrow (3),(4),(5),(6)$. Alternatively to \cite{gioan2008circuit}, the proofs are similar to the above one by restricting to circuit reversals only, then restricting to totally cyclic reorientations only, and then taking duals.

\smallskip

$\bullet$ $(5)\Rightarrow (1)$ assuming $M$ has no loops. 
Suppose $M$ is not regular. Then it has a minor $M/A\setminus B$ that is isomorphic to $U_{2,4}$.
By Corollary \ref{coro:CCMO_red}, it suffices to show that there are two cocircuit minimal acyclic orientations in the same (circuit-)cocircuit reversal acyclic class.
Up to enlarging $A$, we can assume that $M/A$ is loopless and that the rank of $M/A$ equals $2$. 
Since $M$ is loopless, up to reorientation, we can assume that $M$ and $M/A$ are acyclic.
Let $C$ and $D$ be the two positive cocircuits of $M/A$ (which are also cocircuits of $M$).
Thus, $-_CM$ and $-_DM$ are in the same acyclic reversal class.
Denote $S=C\triangle D$.
Since the rank of $M/A$ is $2$, any cocircuit of $M/A$ is the union of all parallel classes but one.
Since $S$ is the union of  two parallel classes ($C\setminus D$ and $D\setminus C$) of $M/A$, no cocircuit of $M$ is contained in $S$ (otherwise, the complement of $S$ is a parallel class in $M/A$, and reducing parallel classes of $M/A$ yields $U_{2,3}$, a contradiction).
Choose a linear ordering of $E$ such that elements of $S$ are greater than elements of $E\setminus S$, then $S$ does not contain the minimum element of a cocircuit (otherwise it would contain a cocircuit), thus $-_CM$ and $-_DM$ are both cocircuit minimal.

\smallskip

$\bullet$  $(2)\Rightarrow (1)$. Suppose $M$ is not regular.
Then $M'$, the oriented matroid obtained from removing all loops of $M$, is also not regular.
By the implication $(5)\Rightarrow (1)$ for $M'$, there exist distinct acyclic reorientations $-_{A}M'$ and $-_{B}M'$ that are (circuit-)cocircuit reversal equivalent and both (circuit-)cocircuit minimal.
Now $-_{A}M$ and $-_{B}M$ are circuit-cocircuit reversal equivalent reorientations of $M$ that are both circuit-cocircuit minimal.
The implication follows from Corollary \ref{coro:CCMO_red}.

\smallskip

 $\bullet$ $(3)\Rightarrow (1)$. The proof is the same as for $(2)\Rightarrow (1)$ except that, at the end, $-_{A}M$ and $-_{B}M$ are cocircuit reversal equivalent reorientations of $M$ that are both cocircuit minimal.
The implication again follows from Corollary \ref{coro:CCMO_red}.

\smallskip

 $\bullet$ $(4)\Rightarrow (1)$, and $(6)\Rightarrow (1)$ assuming $M$ has no coloops.  The two implications are the dual statements of $(3)\Rightarrow (1)$ and $(5)\Rightarrow (1)$, respectively.

\smallskip

Finally, let us mention that the relations between the implication $(5)\Rightarrow (1)$ and the other ones could also be handled from the decomposition into acyclic and totally cyclic parts, along with the convolution formula for the Tutte polynomial, similarly as in \cite{gioan2008circuit}.
\end{proof}

We end with an open question.
By direct computation, the number of circuit-cocircuit reversal classes and the number of bases differ by a rather large margin for small non-regular oriented matroids.
So, does there exist an absolute constant $K>1$ such that the number of bases of a non-regular oriented matroid is at least $K$ times the number of circuit-cocircuit reversal classes?

\section*{Acknowledgements}

The second author, Chi Ho Yuen, would like to thank Matthew Baker for suggesting the problem, and Spencer Backman for introducing him the work of the first author, Emeric Gioan.
The role of the first author of this paper has been to simplify and extend
a preprint written at the initiative of the second author.

\vspace{-2mm}

\bibliographystyle{plain}

\end{document}